\documentclass[12pt]{amsart}
\advance\hoffset by -0.5 truecm
\usepackage{amsmath}
\usepackage{amssymb}
\usepackage{amsthm}
\newtheorem{Theorem}{Theorem}[section]

\newtheorem{Proposition}[Theorem]{Proposition}
\newtheorem{Conjecture}[Theorem]{Conjecture}
\newtheorem{Definition}[Theorem]{Definition}

\def \dim{{\mbox {dim}}\,}

\def\V{\mbox{Var}}

\def\R\re
\def\V{\bf V}

\def \re{{\mathbb R}}

\def \C{{\mathbb C}}

\def \V{{\bf V}}

\newcommand{\id}{\mathrm{Id}}

\newcommand{\abs}[1]{\lvert #1 \rvert}
\newcommand{\norm}[1]{\lVert #1 \rVert}

\newcommand{\mC}{\mathbb{C}}

\begin{document}
\title[Tensor tomography on surfaces]{Tensor tomography on surfaces}

\author[G.P. Paternain]{Gabriel P. Paternain}
\address{ Department of Pure Mathematics and Mathematical Statistics, University of Cambridge, Cambridge CB3 0WB, UK.}
\email{g.p.paternain@dpmms.cam.ac.uk}

\author[M. Salo]{Mikko Salo}
\address{Department of Mathematics  and Statistics, University of Jyv\"askyl\"a.}
\email{mikko.j.salo@jyu.fi}

\author[G. Uhlmann]{Gunther Uhlmann}
\address{Department of Mathematics, University of Washington and UC Irvine.}
\email{gunther@math.washington.edu}




\begin{abstract} 
We show that on simple surfaces the geodesic ray transform acting on solenoidal symmetric
tensor fields of arbitrary order is injective. This solves a long standing inverse problem in the two-dimensional case.
\end{abstract}

\maketitle

\section{Introduction}

Let $(M,g)$ be a compact orientable two-dimensional manifold with smooth boundary.  
We consider the geodesic ray transform acting on symmetric $m$-tensor fields on $M$. When the metric is Euclidean and $m=0$ this transform reduces to the usual X-ray transform obtained by integrating functions along straight lines. More generally, given a symmetric (covariant) $m$-tensor field $f = f_{i_1 \cdots i_m} \,dx^{i_1} \otimes \cdots \otimes \,dx^{i_m}$ on $M$ we define the corresponding function on $SM$ by 
$$
f(x,v) = f_{i_1 \cdots i_m} v^{i_1} \cdots v^{i_m}.
$$
Here $SM = (x,v) \in TM \,;\, \abs{v} = 1 \}$ is the unit circle bundle. Geodesics going from $\partial M$ into $M$ can be parametrized by $\partial_+ (SM) = \{(x,v) \in SM \,;\, x \in \partial M, \langle v,\nu \rangle \leq 0 \}$ where $\nu$ is the outer unit normal vector to $\partial M$. For $(x,v) \in SM$ we let $t \mapsto \gamma(t,x,v)$ be the geodesic starting from $x$ in direction $v$. We assume that $(M,g)$ is nontrapping, which means that the time $\tau(x,v)$ when the geodesic $\gamma(t,x,v)$ exits $M$ is finite for each $(x,v)$  in $SM$.

The ray transform of $f$ is defined by 
$$
If(x,v) = \int_0^{\tau(x,v)} f(\varphi_t(x,v)) \,dt, \quad (x,v) \in \partial_+(SM),
$$
where $\varphi_t$ denotes the geodesic flow of the Riemannian metric $g$ acting on $SM$.
If $h$ is a symmetric $(m-1)$-tensor field, its inner derivative $dh$ is a symmetric $m$-tensor field defined by $dh=\sigma\nabla h$, where $\sigma$ denotes
symmetrization and $\nabla$ is the Levi-Civita connection. A direct calculation in local coordinates shows that
$$
dh(x,v)= Xh(x,v),
$$
where $X$ is the geodesic vector field associated with $\varphi_t$.
If additionally $h|_{\partial M} = 0$, then one clearly has $I(dh) = 0$. The ray transform on symmetric $m$-tensors is said to be $s$-injective if these are the only elements in the kernel. The terminology arises from the fact that
any tensor field $f$ may be written uniquely as $f=f^s+dh$, where $f^s$
is a symmetric $m$-tensor with zero divergence and $h$ is an $(m-1)$-tensor
with $h|_{\partial M} = 0$ (cf. \cite{Sh}). The tensor fields $f^s$ and
$dh$ are called respectively the {\it solenoidal} and {\it potential} parts
of the tensor $f$. Saying that $I$ is $s$-injective is saying precisely that
$I$ is injective on the set of solenoidal tensors.

In this paper we will assume that $(M,g)$ is {\it simple}, a notion that naturally arises in the context of the boundary rigidity problem \cite{Mi}. We recall that a Riemannian manifold with boundary is said to be simple if
the boundary is strictly convex and given any point $p$ in $M$ the exponential map $\exp_p$ is a diffeomorphism onto $M$. In particular, a simple manifold is nontrapping.

 The next result shows that the ray transform on simple surfaces is $s$-injective for tensors of any rank. This settles a long standing question in the
two-dimensional case (cf. \cite{PS} and \cite[Problem 1.1.2]{Sh}).

\begin{Theorem} \label{theorem_sinjectivity}
Let $(M,g)$ be a simple 2D manifold and let $m \geq 0$. If $f$ is a smooth symmetric $m$-tensor field on $M$ which satisfies $If = 0$, then $f = dh$ for some smooth symmetric $(m-1)$-tensor field $h$ on $M$ with $h|_{\partial M} = 0$. (If $m=0$, then $f=0$.)
\end{Theorem}

The geodesic ray transform is closely related to the boundary rigidity problem of determining a metric on a compact Riemannian manifold from its boundary distance function. See \cite{SU, I} for recent reviews. The case $m=0$, that is, the integration of a function along geodesics, is the linearization of the boundary rigidity problem in a fixed conformal class. The standard X-ray transform, where one integrates a function along straight lines, corresponds to the case of the Euclidean metric and is the basis of medical imaging techniques such as CT and PET. The case of integration along more general geodesics arises in geophysical imaging in determining the inner structure of the Earth since the speed of elastic waves generally increases with depth, thus curving the rays back to the Earth surface. It also arises in ultrasound imaging, where the Riemannian metric models the anisotropic index of refraction. Uniqueness and stability for the case $m=0$ was shown by Mukhometov \cite{Mu} on simple surfaces, and Fredholm type inversion formulas were given in \cite{pestovuhlmann_imrn}.

The case $m=1$ corresponds to the geodesic Doppler transform in which one integrates a vector field along geodesics. 
This transform appears in ultrasound tomography to detect tumors using blood flow measurements and also in non-invasive industrial measurements for reconstructing the velocity of a moving fluid. In the case $m=1$, $s$-injectivity  was shown by Anikonov and Romanov \cite{AR} and stability estimates were proven in \cite{SU2}.

The integration of tensors of order two along geodesics arises as the linearization of the boundary rigidity problem and the linear problem is known as deformation boundary rigidity. Sharafutdinov \cite{Sh1} showed $s$-injectivity in this case for
simple surfaces. Sharafutdinov's proof follows the outline of the proof by Pestov and Uhlmann \cite{PU} of the non-linear boundary rigidity problem for simple surfaces and it is certainly more involved than our proof of Theorem \ref{theorem_sinjectivity}, which is independent of the solution to the non-linear problem. 

The case of tensor fields of rank four
describes the perturbation of travel times of compressional waves  
propagating in slightly anisotropic elastic media; see Chapter 7 of
\cite{Sh}. There are very few general results for tensor fields of order higher  
than two. For the case when the underlying metric is Euclidean, 
Sharafutdinov gave explicit reconstruction formulas for the solenoidal part
\cite{Sh}. For results obtained under curvature restrictions see
\cite{Sh,Pe} and \cite{Sh3/2,D} where non-convex boundaries are also considered.

 The only results without curvature assumptions that we are aware of for tensors of order higher than two are \cite{Sh1/2,SSU}. In \cite{SSU} it is shown that $I^*I$ is a pseudodifferential operator
of order $-1$ on a slightly larger simple manifold.
Moreover it
is elliptic on solenoidal tensor fields. Here $I^*$
denotes the adjoint of $I$ with respect to natural inner products. Thus we can recover the wave front set of a distribution solution of 
$If=0.$  Using this result, the analysis of \cite{SU1,SU2} and Theorem \ref{theorem_sinjectivity} it is straightforward to derive stability estimates.

We give two proofs of Theorem \ref{theorem_sinjectivity}.
These proofs are partially inspired
by the proof of the Kodaira vanishing theorem in Complex Geometry \cite{GH}.
The theorem states that if $M$ is a K\"ahler manifold, $K_{M}$ is its canonical line bundle and $E$ is a {\it positive} holomorphic line bundle, then the groups $H^{q}(M,K_{M}\otimes E)$ vanish for any $q>0$. The positivity of $E$ means that there exists a Hermitian metric on $E$ such that $iF_{\nabla}$ is a positive differential form, where $F_{\nabla}$ is the curvature of the canonical connection $\nabla$ induced by the Hermitian metric. Via an $L^2$ energy identity, this positivity of the curvature implies the vanishing of the relevant harmonic forms.
Tensoring with $E$ will be translated in our setting as introducing an appropriate attenuation given by a suitable connection into the relevant transport equation. This attenuation will play the role of $\nabla$ above and a version of the Pestov identity will play the role of the $L^2$ energy identity.
 Of course, this is just an analogy and 
the technical details are very different in the two settings, but the analogy is powerful enough to provide the key idea for solving the transport equation.
One actually has a choice of different possible connections as attenuations, and this is what produces the two different proofs.
This approach was already employed in \cite{PSU}, but it was surprising to us to discover that it could also be successfully used to solve the tensor tomography problem for simple surfaces. We remark that it is still an open problem to establish Theorem \ref{theorem_sinjectivity} when $\dim M\geq 3$ and $m\geq 2$.

After some preliminaries we provide in Section \ref{sec:pestov} a new point of view on the Pestov identity which makes its derivation quite natural. We also explain here why there is an essential difference between the cases $m=0,1$ and $m\geq 2$. In Section \ref{sec:firstproof}
we prove Theorem \ref{theorem_sinjectivity} choosing as attenuation a primitive of the area form. We give an alternative proof in Section \ref{sec:secondproof} choosing as attenuation the Levi-Civita connection. This alternative proof will also yield a more general result. In order to state it, let $I_{m}$ denote the ray transform acting on symmetric $m$-tensors and let $C^{\infty}_{\alpha}(\partial_+(SM))$ denote the set of functions $h\in C^{\infty}(\partial_{+}(SM))$ such that the function $h_{\psi}(x,v) = h(\varphi_{-\tau(x,-v)}(x,v))$ is smooth on $SM$. In natural $L^2$ inner products, the adjoint of $I_0$ is the operator 
$$I_0^*: C^{\infty}_{\alpha}(\partial_+(SM)) \to C^{\infty}(M), \ \ I_0^* h(x) = \int_{S_x} h_{\psi}(x,v) \,dS_x(v).$$
Here $S_x = \{ (x,v) \in TM \,;\, \abs{v} = 1 \}$ and $dS_x$ is the volume form on $S_x$. For more details see \cite{PU}, where it is also proved that $I_0^*$ is surjective on any simple manifold.

\begin{Theorem} Let $(M,g)$ be a compact nontrapping surface with strictly convex smooth boundary. Suppose in addition that $I_0$ and $I_1$ are $s$-injective and that $I_0^*$ is surjective. If $f$ is a smooth symmetric $m$-tensor field on $M$ which satisfies $I_{m}f = 0$, then $f = dh$ for some smooth symmetric $(m-1)$-tensor field $h$ on $M$ with $h|_{\partial M} = 0$.
\label{thm:extension}
\end{Theorem}

Given this result, it seems natural to conjecture that $s$-injectivity on tensors should hold on nontrapping surfaces.

\begin{Conjecture} Let $(M,g)$ be a compact nontrapping surface with strictly convex boundary. If $f$ is a smooth symmetric $m$-tensor field on $M$ which satisfies $If = 0$, then $f = dh$ for some smooth symmetric $(m-1)$-tensor field $h$ on $M$ with $h|_{\partial M} = 0$. (If $m=0$, then $f=0$.)
\label{thm:conjecture}
\end{Conjecture}

\subsection*{Acknowledgements}

M.S. was supported in part by the Academy of Finland, and G.U. was partly supported by NSF and a Rothschild Distinguished Visiting Fellowship at the Isaac Newton Institute. The authors would like to express their gratitude to the Newton Institute and the organizers of the program on Inverse Problems in 2011 where this work was carried out. We also thank the referees for their constructive and useful comments.

\section{Preliminaries}\label{sec:prelim}

Let $(M,g)$ be a compact oriented two dimensional Riemannian manifold with smooth boundary
$\partial M$. As usual $SM$ will denote the unit circle bundle which is a compact 3-manifold with boundary given by $\partial(SM)=\{(x,v)\in SM:\;x\in \partial M\}$.

Let $X$ denote the vector field associated with the geodesic flow $\varphi_{t}$.
Since $M$ is assumed oriented there is a circle action on the fibers of $SM$ with infinitesimal generator $V$ called the {\it vertical vector field}. It is possible to complete the pair $X,V$ to a global frame
of $T(SM)$ by considering the vector field $X_{\perp}:=[X,V]$. There are two additional structure equations given by $X=[V,X_{\perp}]$ and $[X,X_{\perp}]=-KV$
where $K$ is the Gaussian curvature of the surface. Using this frame we can define a Riemannian metric on $SM$ by declaring $\{X,X_{\perp},V\}$ to be an orthonormal basis and the volume form of this metric will be denoted by $d\Sigma^3$. The fact that $\{ X, X_{\perp}, V \}$ are orthonormal together with the commutator formulas implies that the Lie derivative of $d\Sigma^3$ along the three vector fields vanishes, in ther words, the three vector fields preserve the volume form $d\Sigma^3$. We refer the reader to \cite{SiTh} for details on the assertions in this paragraph.

It will be convenient for later purposes and for completness sake to explicitly write down the three vector fields locally. We can always choose isothermal coordinates $(x_{1},x_{2})$ so that the metric
can be written as $ds^2=e^{2\lambda}(dx_{1}^2+dx_{2}^2)$ where $\lambda$ is a smooth
real-valued function of $x=(x_{1},x_{2})$. This gives coordinates $(x_{1},x_{2},\theta)$ on $SM$ where
$\theta$ is the angle between a unit vector $v$ and $\partial/\partial x_{1}$.
In these coordinates we may write $V=\partial/\partial\theta$ and
\[X=e^{-\lambda}\left(\cos\theta\frac{\partial}{\partial x_{1}}+
\sin\theta\frac{\partial}{\partial x_{2}}+
\left(-\frac{\partial \lambda}{\partial x_{1}}\sin\theta+\frac{\partial\lambda}{\partial x_{2}}\cos\theta\right)\frac{\partial}{\partial \theta}\right),\]
\[X_{\perp}=-e^{-\lambda}\left(-\sin\theta\frac{\partial}{\partial x_{1}}+
\cos\theta\frac{\partial}{\partial x_{2}}-
\left(\frac{\partial \lambda}{\partial x_{1}}\cos\theta+\frac{\partial \lambda}{\partial x_{2}}\sin\theta\right)\frac{\partial}{\partial \theta}\right).\]

Given functions $u,v:SM\to \C$ we consider the
inner product
\[(u,v) =\int_{SM}u\bar{v}\,d\Sigma^3.\]
Since $X, X_{\perp}, V$ are volume preserving we have $(Vu,v) = -(u,Vv)$ for $u, v \in C^{\infty}(SM)$, and if additionally $u|_{\partial(SM)} = 0$ or $v|_{\partial(SM)} = 0$ then also $(Xu,v) = -(u,Xv)$ and $(X_{\perp} u, v) = -(u, X_{\perp} v)$.

The space $L^{2}(SM)$ decomposes orthogonally
as a direct sum
\[L^{2}(SM)=\bigoplus_{k\in\mathbb Z}H_{k}\]
where $H_k$ is the eigenspace of $-iV$ corresponding to the eigenvalue $k$.
A function $u\in L^{2}(SM)$ has a Fourier series expansion
\[u=\sum_{k=-\infty}^{\infty}u_{k},\]
where $u_{k}\in H_k$. Also $\|u\|^2=\sum\|u_k\|^2$, where $\|u\|^2=(u,u)^{1/2}$. The even and odd parts of $u$ with respect to velocity are given by 
$$
u_+ := \sum_{k \text{ even}} u_k, \qquad u_- := \sum_{k \text{ odd}} u_k.
$$
In the $(x,\theta)$-coordinates previously introduced we may write
\[u_{k}(x,\theta)=\left( \frac{1}{2\pi} \int_0^{2\pi} u(x,t) e^{-ikt} \,dt \right) e^{ik\theta}=\tilde{u}_{k}(x)e^{ik\theta}.\]
Observe that for $k\geq 0$, $u_k$ may be identified with a section of the $k$-th tensor
power of the canonical line bundle; the identification takes $u_k$ into $\tilde{u}_{k}e^{k\lambda}(dz)^k$ where
$z=x_{1}+ix_{2}$.

\begin{Definition} A function $u:SM\to\C$ is said to be holomorphic
if $u_{k}=0$ for all $k<0$. Similarly, $u$ is said to be antiholomorphic if $u_{k}=0$ for all $k>0$.
\end{Definition}

We will also employ the fiberwise Hilbert transform $H: C^{\infty}(SM) \to C^{\infty}(SM)$, defined in terms of Fourier coefficients as 
$$
Hu_k := -i \,\text{sgn}(k) u_k.
$$
Here $\text{sgn}(k)$ is the sign of $k$, with the convention $\text{sgn}(0) = 0$. Thus, $u$ is holomorphic iff $(\id - iH)u = u_0$ and antiholomorphic iff $(\id + iH)u = u_0$. The commutator of the Hilbert transform and the geodesic vector field was computed in \cite{PU}:

\begin{Proposition} \label{prop:hxcommutator}
If $u \in C^{\infty}(SM)$, then $[H,X]u = X_{\perp} u_0 + (X_{\perp} u)_0$.
\end{Proposition}

The next proposition concerns holomorphic solutions of transport equations. In the case of simple surfaces this was proved in \cite[Proposition 5.1]{SaU}, however here we note that it is true whenever the ray transform on $0$-tensors and $1$-tensors is $s$-injective.

\begin{Proposition} Let $(M,g)$ be a compact nontrapping surface with strictly convex boundary, and assume that $I_0$ and $I_1$ are $s$-injective. If $f$ is a smooth
holomorphic (antiholomorphic) function on $SM$ and if $u\in C^{\infty}(SM)$
satisfies
\[Xu=-f \text{ in } SM, \quad u|_{\partial(SM)} = 0,\]
then $u$ is holomorphic (antiholomorphic) and $u_0=0$.
\label{prop:gm}
\end{Proposition}
\begin{proof}
Suppose that $f$ is holomorphic (the antiholomorphic case is analogous). Using Proposition \ref{prop:hxcommutator}, the function $w = (\id - iH)u \in C^{\infty}(SM)$ satisfies 
$$
Xw = (\id - iH)Xu + i[H,X]u = -(\id - iH)f + i X_{\perp} u_0 + i (X_{\perp} u)_0.
$$
But $(\id - iH)f = f_0$, so splitting in even and odd parts gives 
$$
Xw_- = -f_0 + i(X_{\perp} u)_0, \quad Xw_+ = i X_{\perp} u_0.
$$
Since $u|_{\partial(SM)} = 0$, we have $w_{\pm}|_{\partial(SM)} = 0$. Using this boundary condition we have $IXw_{\pm} = 0$, so the ray transforms of the $0$-tensor $f_0 - i(X_{\perp} u)_0$ and the $1$-tensor $-iX_{\perp} u_0$ vanish. By our injectivity assumption, $f_0 - i(X_{\perp} u)_0 = 0$ and $-iX_{\perp} u_0 = 0$ (note that $-iX_{\perp} u_0$ corresponds to a solenoidal $1$-form). In particular 
$$
Xw_{\pm} = 0.
$$
By the boundary condition for $w_{\pm}$, we obtain $w_{\pm} = 0$. This implies $(\id - iH) u = 0$, so $u$ is holomorphic and $u_0 = 0$.
\end{proof}

Without a doubt, the result that makes everything possible is the existence of holomorphic integrating factors for the transport equation. This is true on nontrapping surfaces whenever $I_0^*$ is surjective. The special case of simple surfaces was covered in \cite[Theorem 4.1]{PSU}, following a similar result for $0$-forms in \cite{SaU}.

\begin{Proposition} \label{thm_holomorphic_integrating_factors}
Let $(M,g)$ be a compact nontrapping surface with strictly convex boundary, and assume that $I_0^*: C^{\infty}_{\alpha}(\partial_+(SM)) \to C^{\infty}(M)$ is surjective. If $A_j \,dx^j$ is a smooth $1$-form on $M$ and $A(x,v) = A_j v^j$, then there exist a holomorphic $w \in C^{\infty}(SM)$ and an antiholomorphic $\tilde{w} \in C^{\infty}(SM)$ such that $Xw = X\tilde{w} = -A$.
\end{Proposition}
\begin{proof}
We only do the holomorphic case (the other case is analogous), and prove that there is a holomorphic $w \in C^{\infty}(SM)$ with $Xw = -A$. First note that the nontrapping condition implies that $(M,g)$ is simply connected.
This follows from a well known fact in Riemannian geometry: a compact connected and non-contractible Riemannian manifold with strictly convex boundary must have a closed geodesic in its interior \cite{T}.
In our case this violates the assumption of being nontrapping unless $M$ is contractible and hence simply connected.

Since $M$ is simply connected, there is a Hodge decomposition $A_j \,dx^j = da + \star db$ for some $a, b \in C^{\infty}(M)$ ($\star$ is the Hodge star operator) and we have $A = Xa + X_{\perp} b$. Replacing $w$ by $w-a$, it is enough to consider the case where $A = X_{\perp} b$.

Let us try a solution of the form $w = (\id + iH) \hat{w}$ where $\hat{w} \in C^{\infty}(SM)$ is even. By Proposition \ref{prop:hxcommutator}, 
$$
Xw = (\id + iH) X \hat{w} - i[H,X] \hat{w} = (\id + iH) X \hat{w} - i X_{\perp} \hat{w}_0.
$$
Now it is sufficient to find $\hat{w}$ even with $X \hat{w} = 0$ and $\hat{w}_0 = -ib$. Using the surjectivity of $I_0^*$, there is some $h \in C^{\infty}_{\alpha}(\partial_+(SM))$ with $I_0^* h = -2\pi ib$. But if $w' \in C^{\infty}(SM)$ is the function with $X w' = 0$ in $SM$ and $w'|_{\partial_+(SM)} = h$, we have $(w')_0 = \frac{1}{2\pi} I_0^* h = -ib$. It is enough to take $\hat{w} = w'_+$.
\end{proof}

We will also make use of a regularity result from \cite[Proposition 5.2]{PSU}.

\begin{Proposition} \label{prop_transport_regularity}
Let $(M,g)$ be a compact nontrapping manifold with strictly convex boundary, let $f \in C^{\infty}(SM)$, and define  
$$
u(x,v) := \int_0^{\tau(x,v)} f(\varphi_t(x,v)) \,dt, \quad (x,v) \in SM.
$$
If $If = 0$, then $u \in C^{\infty}(SM)$.
\end{Proposition}

We conclude this section by explaining
the identification between real-valued symmetric $m$-tensor fields
and certain smooth functions on $SM$.
Given such a tensor $f = f_{i_1 \cdots i_m} \,dx^{i_1} \otimes \cdots \otimes \,dx^{i_m}$ we consider the corresponding function on $SM$ (henceforth referred to as the restriction) defined by 
\[f(x,v) = f_{i_1 \cdots i_m} v^{i_1} \cdots v^{i_m}.\]
Then clearly
\[f=\sum_{k=-m}^{m}f_k\]
where $\bar{f}_{k}=f_{-k}$. Moreover if $m$ is even (resp. odd) all
the odd (resp. even) Fourier coefficients vanish.

Conversely suppose that we are given a smooth real-valued function
$f\in C^{\infty}(SM)$ such that $f_{k}=0$ for $|k|\geq m+1$.
Suppose in addition that if $m$ is even (resp. odd)
then $f=f_{+}$ (resp. $f=f_{-}$). Since $f$ is real-valued $\bar{f}_{k}=f_{-k}$. For each $k\geq 1$, the function $f_{-k}+f_{k}$ gives rise
to a unique real-valued symmetric $k$-tensor $F_k$ whose restriction
to $SM$ is precisely $f_{-k}+f_{k}$. This can be seen as follows: recall that a smooth
element $f_{k}$ can be identified with a section of $T^*M^{\otimes k}$ hence, its real
part defines a symmetric $k$-tensor. (For $k=0$, $\bar{f}_{0}=f_{0}$ is obviously a real-valued $0$-tensor.) More explicitly, in the coordinates $(x,\theta)$, given $f_{k}=\tilde{f}_{k}e^{ik\theta}$ we define
\[F_{k}:=2\Re(\tilde{f}_{k}e^{k\lambda}(dz)^k).\]
It is straightforward to check that these local expressions glue together to give a real-valued symmetric $k$-tensor whose restriction to $SM$ is $f_{-k}+f_{k}$.

By tensoring with the metric tensor $g$ and symmetrizing it is possible
to raise the degree of a symmetric tensor by two. Hence if $\sigma$ denotes symmetrization, $\alpha F_{k}:=\sigma(F_k\otimes g)$
will be a symmetric tensor of degree $k+2$ whose restriction to
$SM$ is again $f_{k}+f_{-k}$ since $g$ restricts as the constant function $1$ 
to $SM$. Now consider the symmetric $m$-tensor
\[F:=\sum_{i=0}^{[m/2]}\alpha^{i}F_{m-2i}.\]
It is easy to check that the restriction of $F$ to $SM$ is precisely $f$.

\section{Pestov identity}
\label{sec:pestov}

In this section we consider the Pestov identity, which is the basic energy identity that has been used since the work of Mukhometov \cite{Mu} in most injectivity proofs of ray transforms in the absence of real-analyticity or special symmetries. The Pestov identity often appears in a somewhat ad hoc way, but here we give a new point of view which makes its derivation more transparent.

The easiest way to motivate the Pestov identity is to consider the injectivity of the ray transform on functions. The first step is to recast the injectivity problem as a uniqueness question for the partial differential operator $P$ on $SM$ where 
$$
P := VX.
$$
This involves a standard reduction to the transport equation.

\begin{Proposition}
Let $(M,g)$ be a compact oriented nontrapping surface with strictly convex smooth boundary. The following statements are equivalent.
\begin{enumerate}
\item[(a)]
The ray transform $I: C^{\infty}(M) \to C(\partial_+(SM))$ is injective.
\item[(b)]
Any smooth solution of $Pu = 0$ in $SM$ with $u|_{\partial(SM)} = 0$  is identically zero.
\end{enumerate}
\end{Proposition}
\begin{proof}
Assume that the ray transform is injective, and let $u \in C^{\infty}(SM)$   solve $Pu = 0$   in $SM$   with $u|_{\partial(SM)} = 0$. This implies that $Xu = -f$   in $SM$ for some smooth $f$   only depending on $x$, and we have $0 = u|_{\partial_+(SM)} = If$. Since $I$   is injective one has $f = 0$   and thus $Xu = 0$, which implies $u = 0$   by the boundary condition.

Conversely, assume that the only smooth solution of $Pu = 0$   in $SM$ which vanishes on $\partial(SM)$ is zero. Let  $f \in C^{\infty}(M)$ be a function with $If = 0$, and define the function 
$$
u(x,v) := \int_0^{\tau(x,v)} f(\gamma(t,x,v)) \,dt, \quad (x,v) \in SM.
$$
This function satisfies the transport equation $Xu = -f$ in $SM$ and $u|_{\partial(SM)} = 0$ since $If = 0$, and also $u \in C^{\infty}(SM)$ by Proposition \ref{prop_transport_regularity}. Since $f$ only depends on $x$ we have $Vf = 0$, and consequently $Pu = 0$ in $SM$ and $u|_{\partial(SM)} = 0$. It follows that $u = 0$ and also $f = -Xu = 0$.
\end{proof}

We now focus on proving a uniqueness statement for solutions of $Pu = 0$   in $SM$. For this it is convenient to express $P$ in terms of its self-adjoint and skew-adjoint parts in the $L^2(SM)$ inner product as 
$$
P = A + iB, \quad A := \frac{P+P^*}{2},  \   \   B := \frac{P-P^*}{2i}.
$$
Here the formal adjoint $P^*$    of $P$   is given by 
$$
P^* := XV.
$$
In fact, if $u \in C^{\infty}(SM)$   with $u|_{\partial(SM)} = 0$, then 
\begin{align}
\norm{Pu}^2 &= ((A+iB)u, (A+iB)u) = \norm{Au}^2 + \norm{Bu}^2 + i(Bu,Au) - i(Au,Bu)  \label{p_ab_computation} \\
 &= \norm{Au}^2 + \norm{Bu}^2 + (i[A,B]u, u). \notag
\end{align}
This computation suggests to study the commutator $i[A,B]$. We note that the argument just presented is typical in the proof of $L^2$ Carleman estimates \cite{H}.

By the definition of $A$   and $B$ it easily follows that $i[A,B] = \frac{1}{2} [P^*, P]$. By the commutation formulas for $X$, $X_{\perp}$ and $V$, this commutator may be expressed as 
\begin{align*}
[P^*, P] &= XVVX - VXXV = VXVX + X_{\perp} VX - VXVX - VX X_{\perp} \\
 &= V[X_{\perp}, X] - X^2 = -X^2 + VKV.
\end{align*}
Consequently 
$$
([P^*, P]u, u) = \norm{Xu}^2 - (KVu,Vu).
$$
If the curvature $K$ is nonpositive, then $[P^*,P]$ is positive semidefinite. More generally, one can try to use the other positive terms in \eqref{p_ab_computation}. Note that 
$$
 \norm{Au}^2 + \norm{Bu}^2 = \frac{1}{2}(\norm{Pu}^2 + \norm{P^* u}^2).
$$
The identity \eqref{p_ab_computation} may then be expressed as 
\begin{align*}
\norm{Pu}^2 &= \norm{P^* u}^2 + ([P^*,P]u, u).
\end{align*}
Moving the term $\norm{Pu}^2$ to the other side, we have proved the version of the Pestov identity which is most suited for our purposes. The main point in this proof was that the Pestov identity boils down to a standard $L^2$ estimate based on separating the self-adjoint and skew-adjoint parts of $P$ and on computing one commutator, $[P^*, P]$.

\begin{Proposition} \label{prop_pestov_standard}
If $(M,g)$ is a compact oriented surface with smooth boundary, then 
$$
\norm{XVu}^2 - (KVu,Vu) + \norm{Xu}^2 - \norm{VXu}^2 = 0
$$
for any $u \in C^{\infty}(SM)$ with $u|_{\partial(SM)} = 0$.
\end{Proposition}

It is well known (cf. proof of \cite[Proposition 7.2]{DKSU}) that on a simple surface, one has 
$$
\norm{XVu}^2 - (KVu,Vu) \geq 0, \quad u \in C^{\infty}(SM), \ u|_{\partial(SM)} = 0.
$$
Also, if $Xu = -f$ where $f = f_0 + f_1 + f_{-1}$ is the sum of a $0$-form and $1$-form, we have 
$$
\norm{Xu}^2 - \norm{VXu}^2 = \norm{f_0}^2 \geq 0.
$$
These two facts together with the Pestov identity give the standard proof of $s$-injectivity of the ray transform for $0$-forms and $1$-forms on simple surfaces. It is easy to see where this proof breaks down if $m \geq 2$: the Fourier expansion $f = \sum_{k=-m}^m f_k$ implies  
$$
\norm{Xu}^2 - \norm{VXu}^2 = \norm{f_0}^2 - \sum_{2 \leq \abs{k} \leq m} (k^2-1) \norm{f_k}^2.
$$
This term may be negative, and the Pestov identity may not give useful information unless there is some extra positivity like a curvature bound.

Finally, we consider the Pestov identity in the presence of attenuation given by $A(x,v) = A_j(x) v^j$ where $A_j \,dx^j$ is a purely imaginary $1$-form on $M$. We write $A$   both for the $1$-form and the function on $SM$. The geometric interpretation is that $d+A$ is a unitary connection on the trivial bundle $M \times \mC$, and its curvature is the $2$-form 
$$
F_A := dA + A \wedge A=dA.
$$
Then $\star F_A$ is a function on $M$ where $\star$ is the Hodge star. We consider the operator 
$$
P := V(X+A).
$$
Since $\bar{A} = -A$, the formal adjoint of $P$   in the $L^2(SM)$   inner product is 
$$
P^* = (X+A)V.
$$
The same argument leading to Proposition \ref{prop_pestov_standard}, based on computing the commutator $[P^*, P]$, gives the following Pestov identity proved also in \cite[Lemma 6.1]{PSU}.

\begin{Proposition} \label{prop_pestov_attenuation}
If $(M,g)$ is a compact oriented surface with smooth boundary and if $A$ is a purely imaginary $1$-form on $M$, then 
$$
\norm{(X+A)Vu}^2 - (KVu,Vu) + \norm{(X+A)u}^2 - \norm{V(X+A)u}^2 + (\star F_A Vu,u) = 0
$$
for any $u \in C^{\infty}(SM)$ with $u|_{\partial(SM)} = 0$.
\end{Proposition}

Using the Fourier expansion of $u$, the last term in the identity is given by 
$$
\sum_{k=-\infty}^{\infty} ik(\star F_A u_k, u_k)
$$
This shows that if $u$   is holomorphic and $i\star F_A > 0$, or if $u$ is antiholomorphic and $i\star F_A < 0$, one gains an additional positive term in the Pestov identity. This will be crucial in absorbing negative contributions from the term $\norm{(X+A)u}^2 - \norm{V(X+A)u}^2$ when proving $s$-injectivity on tensor fields.

\section{First proof}
\label{sec:firstproof}

The proof of $s$-injectivity of the ray transform reduces to proving the next result. We say that $f \in C^{\infty}(SM)$ has degree $m$ if $f_k = 0$ for $\abs{k} \geq m+1$.

\begin{Proposition} \label{proposition_finitedegree}
Let $(M,g)$ be a simple surface, and assume that $u \in C^{\infty}(SM)$ satisfies $Xu = -f$ in $SM$ with $u|_{\partial(SM)} = 0$. If $f \in C^{\infty}(SM)$ has degree $m \geq 1$, then $u$ has degree $m-1$. If $f$ has degree $0$, then $u = 0$.
\end{Proposition}

\begin{proof}[Proof of Theorem \ref{theorem_sinjectivity}]
Let $f$ be a symmetric $m$-tensor field on $SM$ and suppose that $If = 0$. We write 
$$
u(x,v) := \int_0^{\tau(x,v)} f(\varphi_t(x,v))\,dt, \quad (x,v) \in SM.
$$
Then $u|_{\partial(SM)} = 0$, and also $u \in C^{\infty}(SM)$ by Proposition \ref{prop_transport_regularity}.

Now $f$ has degree $m$, and $u$ satisfies $Xu = -f$ in $SM$ with $u|_{\partial(SM)} = 0$. Proposition \ref{proposition_finitedegree} implies that $u$ has degree $m-1$ (and $u=0$ if $m=0$). We let $h := -u$.
As we explained in Section \ref{sec:prelim}, $h$ gives rise to a symmetric $(m-1)$-tensor still denoted by $h$. Since
$X(h)=f$, this implies that $dh$ and $f$ agree when restricted to $SM$ and since they are both symmetric
tensors of the same degree it follows that $dh=f$. This proves the theorem.
\end{proof}

Proposition \ref{proposition_finitedegree} is in turn an immediate consequence of the next two results. Note that these are generalizations of Proposition \ref{prop:gm}.

\begin{Proposition} \label{proposition_onesidedfourier1}
Let $(M,g)$ be a simple surface, and assume that $u \in C^{\infty}(SM)$ satisfies $Xu = -f$ in $SM$ with $u|_{\partial(SM)} = 0$. If $m \geq 0$ and if $f \in C^{\infty}(SM)$ is such that $f_k = 0$ for $k \leq -m-1$, then $u_k = 0$ for $k \leq -m$.
\end{Proposition}

\begin{Proposition} \label{proposition_onesidedfourier2}
Let $(M,g)$ be a simple surface, and assume that $u \in C^{\infty}(SM)$ satisfies $Xu = -f$ in $SM$ with $u|_{\partial(SM)} = 0$. If $m \geq 0$ and if $f \in C^{\infty}(SM)$ is such that $f_k = 0$ for $k \geq m+1$, then $u_k = 0$ for $k \geq m$.
\end{Proposition}

We will only prove Proposition \ref{proposition_onesidedfourier1}, the proof of the other result being completely analogous. As a preparation, following Guillemin and Kazhdan in \cite{GK} we introduce the 
first order elliptic operators $\eta_{\pm}:C^{\infty}(SM)\to C^{\infty}(SM)$ given by
\[\eta_{+}:=(X+iX_{\perp})/2,\;\;\;\;\;\;\eta_{-}:=(X-iX_{\perp})/2.\]
Clearly $X=\eta_{+}+\eta_{-}$. Let $\Omega_{k}:=C^{\infty}(SM)\cap H_{k}$. The commutation relations $[-iV,\eta_{+}]=\eta_{+}$ and
$[-iV,\eta_{-}]=-\eta_{-}$ imply that $\eta_{\pm}:\Omega_{k}\to \Omega_{k \pm 1}$. If $A(x,v) = A_j(x) v^j$ where $A$ is a purely imaginary $1$-form on $M$, we also split $A = A_+ + A_-$ where $A_{\pm} \in \Omega_{\pm 1}$ and write 
$$
\mu_+ := \eta_+ + A_+, \quad \mu_- := \eta_- + A_-.
$$
Note that $X+A = \mu_+ + \mu_-$. One has $\mu_{\pm}: \Omega_k \to \Omega_{k \pm 1}$, and 
\begin{equation} \label{mu_adjoint}
(\mu_{\pm} u,v) = -(u,\mu_{\mp} v), \quad u,v \in C^{\infty}(SM) \text{ with } u|_{\partial(SM)} = 0 \text{ or } v|_{\partial(SM)} = 0.
\end{equation}
The following commutator formula was proved in \cite[Lemma 4.3]{P}: 
\begin{equation} \label{mu_commutator}
[\mu_+,\mu_-]u = \frac{i}{2}(KVu + (\star F_A)u).
\end{equation}

\begin{proof}[Proof of Proposition \ref{proposition_onesidedfourier1}]
Assume that $u \in C^{\infty}(SM)$ satisfies $Xu = -f$ in $SM$ and $u|_{\partial(SM)} = 0$, where $f_k = 0$ for $k \leq -m-1$. Splitting in even and odd parts with respect to velocity gives that $X u_{\mp} = -f_{\pm}$ in $SM$, $u_{\mp}|_{\partial(SM)} = 0$. Therefore it is enough to prove the result in two cases: when $f$ is even (so $m$ is even) and $u$ is odd, and when $f$ is odd (so $m$ is odd) and $u$ is even. The proofs for these cases are very similar, and we will only consider the first case.

Assume that $f$ is even, $m$ is even, and $u$ is odd. Let $\omega_g$ be the area form of $(M,g)$ and choose a real valued $1$-form $\varphi$ with $d\varphi = \omega_g$. Consider the unitary connection 
$$
A(x,v) := is \varphi_j(x) v^j
$$
where $s > 0$ is a fixed number to be chosen later. Then $i \star F_A = -s$. By Proposition \ref{thm_holomorphic_integrating_factors}, there exists a holomorphic $w \in C^{\infty}(SM)$ satisfying $Xw = -i\varphi$. We may assume that $w$ is even. The functions $\tilde{u} := e^{sw} u$ and $\tilde{f} := e^{sw} f$ then satisfy 
$$
(X+A) \tilde{u} = -\tilde{f} \text{ in } SM, \quad \tilde{u}|_{\partial(SM)} = 0.
$$
Using that $e^{sw}$ is holomorphic, we have $\tilde{f}_k = 0$ for $k \leq -m-1$. Also, since $e^{sw}$ is even, $\tilde{f}$ is even and $\tilde{u}$ is odd. We now define 
$$
v := \sum_{k=-\infty}^{-m-1} \tilde{u}_k.
$$
Then $v \in C^{\infty}(SM)$, $v|_{\partial(SM)} = 0$, and $v$ is odd. Also, $((X+A)v)_k = \mu_+ v_{k-1} + \mu_- v_{k+1}$. If $k \leq -m-2$ one has $((X+A)v)_k = ((X+A)\tilde{u})_k = 0$, and if $k \geq -m+1$ then $((X+A)v)_k = 0$ since $v_j = 0$ for $j \geq -m$. Also $((X+A)v)_{-m-1} = 0$ because $v$ is odd. Therefore the only nonzero Fourier coefficient is $((X+A)v)_{-m}$, and 
$$
(X+A)v = \mu_+ v_{-m-1} \text{ in } SM, \quad v|_{\partial(SM)} = 0.
$$

We apply the Pestov identity in Proposition \ref{prop_pestov_attenuation} with attenuation $A$ to $v$, so that 
$$
\norm{(X+A)Vv}^2 - (KVv, Vv) + (\star F_A Vv, v) + \norm{(X+A)v}^2 - \norm{V(X+A)v}^2 = 0.
$$
It was proved in \cite[Lemma 6.5]{PSU} that if $(M,g)$ is simple and $v|_{\partial(SM)} = 0$, then 
\begin{equation} \label{pestov_estimate1}
\norm{(X+A)Vv}^2 - (KVv, Vv) \geq 0.
\end{equation}
We also have 
\begin{equation} \label{pestov_estimate2}
(\star F_A Vv, v) = -\sum_{k=-\infty}^{-m-1} i\abs{k}(\star F_A v_k, v_k) = s \sum_{k=-\infty}^{-m-1} \abs{k} \norm{v_k}^2.
\end{equation}
For the remaining two terms, we compute 
$$
\norm{(X+A)v}^2 - \norm{V(X+A)v}^2 = \norm{\mu_+ v_{-m-1}}^2 - m^2 \norm{\mu_+ v_{-m-1}}^2.
$$
If $m = 0$, then this expression is nonnegative and we obtain from the energy identity that $v = 0$. Assume from now on that $m \geq 2$. Using \eqref{mu_adjoint}, \eqref{mu_commutator}, and the fact that $v_k|_{\partial(SM)} = 0$ for all $k$, we have 
\begin{align*}
\norm{\mu_+ v_k}^2 &= \norm{\mu_- v_k}^2 + \frac{i}{2}(K V v_k + (\star F_A) v_k, v_k) \\
 &= \norm{\mu_- v_k}^2 - \frac{s}{2} \norm{v_k}^2 - \frac{k}{2} (K v_k, v_k).
\end{align*}
If $k \leq -m-1$ we also have 
$$
\mu_+ v_{k-1} + \mu_- v_{k+1} = ((X+A)v)_k = 0.
$$
We thus obtain 
\begin{align*}
 &\norm{(X+A)v}^2 - \norm{V(X+A)v}^2 = -(m^2-1) \norm{\mu_+ v_{-m-1}}^2 \\
 &= -(m^2-1) \left[ \norm{\mu_- v_{-m-1}}^2 - \frac{s}{2} \norm{v_{-m-1}}^2 + \frac{m+1}{2} (K v_{-m-1}, v_{-m-1}) \right] \\
 &= -(m^2-1) \left[ \norm{\mu_+ v_{-m-3}}^2 - \frac{s}{2} \norm{v_{-m-1}}^2 + \frac{m+1}{2} (K v_{-m-1}, v_{-m-1}) \right] \\
 &= -(m^2-1) \bigg[ \norm{\mu_- v_{-m-3}}^2 - \frac{s}{2} (\norm{v_{-m-1}}^2 + \norm{v_{-m-3}}^2) \\
 &\qquad \qquad \qquad + \frac{m+1}{2} (K v_{-m-1}, v_{-m-1}) + \frac{m+3}{2} (K v_{-m-3}, v_{-m-3}) \bigg].
\end{align*}
Continuing this process, and noting that $\mu_- v_k \to 0$ in $L^2(SM)$ as $k \to -\infty$ (which follows since $\mu_- v \in L^2(SM)$), we obtain 
\begin{equation} \label{pestov_estimate3}
\norm{(X+A)v}^2 - \norm{V(X+A)v}^2 = \frac{m^2-1}{2} s \sum \norm{v_k}^2 - \frac{m^2-1}{2} \sum \abs{k} (Kv_k, v_k).
\end{equation}

Collecting \eqref{pestov_estimate1}--\eqref{pestov_estimate3} and using them in the Pestov identity implies that 
$$
0 \geq \frac{m^2-1}{2} s \sum \norm{v_k}^2 + \left( s - \frac{m^2-1}{2} \sup_M K \right) \sum \abs{k} \norm{v_k}^2.
$$
If we choose $s > \frac{m^2-1}{2} \sup_M K$, then both terms above are nonnegative and therefore have to be zero. It follows that $v = 0$, so $\tilde{u}_k = 0$ for $k \leq -m-1$ and also $u_k = 0$ for $k \leq -m-1$ since $u = e^{-sw} \tilde{u}$ where $e^{-sw}$ is holomorphic.
\end{proof}

\section{Second proof}
\label{sec:secondproof}
In this section we provide a short alternative proof to the key proposition in the injectivity result.

\begin{proof}[Second proof of Proposition \ref{proposition_onesidedfourier1}]
Suppose that $u$   is a smooth solution of $Xu = -f$   in $SM$ where $f_k = 0$ for $k \leq -m-1$ and $u|_{\partial(SM)} = 0$. We choose a nonvanishing function $h \in \Omega_m$ and define  
$$
A := -h^{-1} Xh.
$$
Observe that since $h\in\Omega_m$, then $h^{-1}\in\Omega_{-m}$. Also $Xh=\eta_{+}h+\eta_{-}h\in\Omega_{m+1}\oplus\Omega_{m-1}$ which implies that $A\in\Omega_{1}\oplus\Omega_{-1}$. It follows that $A$ is the restriction to $SM$ of a (complex-valued) 1-form on $M$.

The function $hu$ solves the problem 
$$
(X + A)(hu) = -hf \text{ in } SM, \quad hu|_{\partial(SM)} = 0.
$$
Note that $hf$ is a holomorphic function. Next we employ a holomorphic integrating factor: by Proposition \ref{thm_holomorphic_integrating_factors} there exists a holomorphic $w \in C^{\infty}(SM)$ with $Xw = A$. The function $e^w h u$ then satisfies 
$$
X(e^w h u) = -e^w h f \text{ in } SM, \quad e^w h u|_{\partial(SM)} = 0.
$$
The right hand side $e^w h f$ is holomorphic. Now Proposition \ref{prop:gm} implies that the solution $e^w h u$ is also holomorphic and $(e^w h u)_0 = 0$. Looking at Fourier coefficients shows that $(h u)_k = 0$   for $k \leq 0$, and therefore $u_k = 0$ for $k \leq -m$ as required.
\end{proof}

Finally, let us explain the choice of $h$ and $A$ in the proof in more detail. Since $M$ is a disk we can consider global isothermal coordinates $(x,y)$ on $M$ such that the metric
can be written as $ds^2=e^{2\lambda}(dx^2+dy^2)$ where $\lambda$ is a smooth
real-valued function of $(x,y)$. As explained in Section \ref{sec:prelim} this gives coordinates $(x,y,\theta)$ on $SM$ where
$\theta$ is the angle between a unit vector $v$ and $\partial/\partial x$. Then $\Omega_m$ consists of all functions $a(x,y) e^{im\theta}$ where $a \in C^{\infty}(M)$. We choose the specific nonvanishing function  
$$
h(x,y,\theta) := e^{im\theta}.
$$ 

Recall that in the $(x,y,\theta)$ coordinates the geodesic vector field $X$ is given by:
\[X=e^{-\lambda}\left(\cos\theta\frac{\partial}{\partial x}+
\sin\theta\frac{\partial}{\partial y}+
\left(-\frac{\partial \lambda}{\partial x}\sin\theta+\frac{\partial\lambda}{\partial y}\cos\theta\right)\frac{\partial}{\partial \theta}\right).\]
The connection $A = -Xh/h$ has the form 
\begin{align*}
A &=ime^{-\lambda}\left(-\frac{\partial \lambda}{\partial y}\cos\theta+\frac{\partial\lambda}{\partial x}\sin\theta\right) = im \left(-\frac{\partial\lambda}{\partial y}dx+\frac{\partial \lambda}{\partial x}dy\right).
\end{align*}
Here as usual we identify $A$ with $A(x,v)$ where $(x,v)\in SM$. This shows that the connection $A$ is essentially the Levi-Civita connection of the metric $g$ on the tensor power bundle $TM^{\otimes m}$, and since $(X+A)h = 0$ we have that $h$ corresponds to a section of the
pull-back bundle $\pi^*(TM^{\otimes m})$ whose covariant derivative along the geodesic vector field vanishes (here $\pi:SM\to M$ is the 
standard projection).

\medskip

\noindent{\it Proof of Theorem \ref{thm:extension}.} An inspection of the proof above
reveals that it also proves Theorem \ref{thm:extension} since we have assumed precisely the hypotheses that allow the use of Propositions \ref{prop:gm} and \ref{thm_holomorphic_integrating_factors}. Also observe that, as pointed out in the proof of Proposition \ref{thm_holomorphic_integrating_factors}, a compact nontrapping surface with strictly convex boundary must be simply connected and hence a disk.

\qed

\end{document}